\documentclass[reqno]{amsart}
\usepackage{amsmath,amssymb,cite} 
\usepackage[mathscr]{euscript}
\theoremstyle{plain}
\newtheorem{theorem}{Theorem}[section]
\newtheorem{proposition}[theorem]{Proposition}
\newtheorem{lemma}[theorem]{Lemma}
\newtheorem{corollary}[theorem]{Corollary}
\theoremstyle{definition}
\newtheorem{assumption}[theorem]{Assumption}

\theoremstyle{remark}

\numberwithin{equation}{section}
\numberwithin{theorem}{section}
\newcommand{\mc}[1]{{\mathcal #1}}
\newcommand{\bb}[1]{{\mathbb #1}}
\newcommand{\ms}[1]{{\mathscr #1}}

\newcommand{\upbar}[1]{\,\overline{\! #1}}

\newcommand{\id}{{1 \mskip -5mu {\rm I}}}
\renewcommand{\epsilon}{\varepsilon}
\renewcommand{\tilde}{\widetilde}
\renewcommand{\hat}{\widehat}
\newcommand{\supp}{\mathop{\rm supp}\nolimits}

\newcommand{\Ent}{\mathop{\rm Ent}\nolimits}

\newcommand{\sce}{\mathop{\rm sc^-\!}\nolimits}

\def\Xint#1{\mathchoice
   {\XXint\displaystyle\textstyle{#1}}%
   {\XXint\textstyle\scriptstyle{#1}}%
   {\XXint\scriptstyle\scriptscriptstyle{#1}}%
   {\XXint\scriptscriptstyle\scriptscriptstyle{#1}}%
   \!\int}
\def\XXint#1#2#3{{\setbox0=\hbox{$#1{#2#3}{\int}$}
     \vcenter{\hbox{$#2#3$}}\kern-.5\wd0}}
\def\dashint{\Xint-}
\def\mm{\textrm{---}}
\def\Xiint#1{\mathchoice
   {\XXiint\displaystyle\textstyle{#1}}%
   {\XXiint\textstyle\scriptstyle{#1}}%
   {\XXiint\scriptstyle\scriptscriptstyle{#1}}%
   {\XXiint\scriptscriptstyle\scriptscriptstyle{#1}}%
   \!\iint}
\def\XXiint#1#2#3{{\setbox0=\hbox{$#1{#2#3}{\iint}$}
     \vcenter{\hbox{$#2#3$}}\kern-.5\wd0}}
\def\dashiint{\Xiint\mm}

\title[LDP for degenerate jump processes]{Donsker-Varadhan asymptotics
  for \\ degenerate jump Markov processes} 

\author[G.\ Basile]{Giada Basile}
\address{Giada Basile \hfill\break \indent
  Dipartimento di Matematica, Universit\`a di Roma `La Sapienza'
  \hfill\break \indent
  P.le Aldo Moro 5, 00185 Roma, Italy}
\email{basile@mat.uniroma1.it}

\author[L.\ Bertini]{Lorenzo Bertini}
\address{Lorenzo Bertini \hfill\break \indent
   Dipartimento di Matematica, Universit\`a di Roma `La Sapienza'
   \hfill\break \indent
   P.le Aldo Moro 5, 00185 Roma, Italy}
 \email{bertini@mat.uniroma1.it}

\begin{document}

\begin{abstract}
  We consider a class of continuous time Markov chains on a compact
  metric space that  admit an invariant measure strictly
  positive on open sets together with  absorbing states.  We prove
  the joint large deviation principle for the empirical measure and
  flow.  Due to the lack of uniform ergodicity, the zero level set of
  the rate function is not a singleton.  As  corollaries, we obtain the
  Donsker-Varadhan rate function for the empirical
  measure and a variational expression of the rate function for the empirical flow.
\end{abstract}


\medskip


\maketitle
\thispagestyle{empty}

\section{Introduction}
\label{s:1}

The energy transport in insulators can be described within a kinetic approach
analogous to  the kinetic theory of gases. At low
temperatures the lattice vibrations, responsible of energy transport, can be modeled  as
a gas of interacting particles (phonons) and their time-dependent distribution function  solves a Boltzmann type equation.
The basic scheme to derive phononic 
Boltzmann equations from the underlying microscopic dynamics is introduced in \cite{Sp}. Following this approach,
in \cite{BOS} a harmonic chain of oscillators perturbed by a conservative weak stochastic noise is analyzed and the following linear Boltzmann equation is derived 
\begin{equation}\begin{split}\label{Be0}
\partial_t W(t,r, k) +  v(k)\partial_r W(t,r,k)
=\int_{\bb T }\;dk'
R(k,k')[W(t,r,k')-W(t,r,k)].
\end{split}\end{equation}
Here $W$ is the energy density distribution of phonons with wave number $k\in\bb T$ (the one dimensional torus), $r$ is the space coordinate, $t$ is the time and $v(k)$ is the velocity of a phonon with wave number $k$ and it is given by the gradient of the dispersion relation.
The scattering kernel $R$ is positive and symmetric. Referring to \cite{BOS} for 
explicit expressions of $R$ and  $v$ and the analogous equation in higher dimensions, 
we point out the following  features.  The velocity  is finite for small  $k$ while  $R$ behaves like $ k^2$ for small $k$, and like $ k'^2$  for small $k'$. This means  that 
phonons with small wave numbers 
travel with finite velocity, but they have low probability to be
scattered, therefore their mean free paths 
have a macroscopic length (ballistic transport).

The equation \eqref{Be0} can be interpreted as the Fokker-Planck equation 
for the Markov process $\left(K(t),
Y(t)\right)$ on $\bb T\times \bb R$, where the wave number 
$K(t)$ is a jump process  and the position $Y(t)$ is an
additive functional of $K$, namely $Y(t)=\int_0^t ds\;
v(K(s))$. In view of the behavior of the kernel $R$ mentioned above, $k=0$ is an absorbing state for the process 
$K(t)$. On the other hand, 
the skeleton associated to $K$  admits an invariant measure $\pi_0$ for which the mean jump time is integrable.
As we prove, this condition implies the ergodicity of the process with $K(0)$ different from $0$. 
Nevertheless,
in dimension one and two the variance of the mean jump time with respect to $\pi_0$ is infinite, so that the standard  central limit theorem for the position $Y(t)$ fails. More precisely, in  one dimension  the position converges to a $3/2$ stable L\'evy process under the proper scaling  \cite{JKO, BB}, while in two  dimensions
it converges to a Brownian motion under an anomalous scaling with logarithmic corrections \cite{Ba}. 
The  purpose of the present paper is to analyze how the degenerate behavior of the kernel $R$ affects the large deviations properties of the process $K$.

In the general context of continuous time Markov processes, the
empirical measure associates to a given trajectory the fraction of
time spent on the different states up to a time $T$.  Under ergodicity
assumptions, the empirical measure converges to the invariant measure.
The corresponding large deviations asymptotic is the content of the
Donsker-Varadhan theorem, with a rate function given by a variational
formula that can be computed explicitly only in the reversible case.
A natural generalization of this framework in the setting of jump
processes takes into account, together with the empirical measure, the
empirical flow which counts the  number of jumps between the
different states per unit of time. We remark that a relevant dynamical observable, the
empirical current, is directly related to the empirical flow.

The joint large deviation asymptotics for the empirical measure and
flow can be derived by contraction from the corresponding result for
the empirical process, which yields the information on arbitrary
sequences of jumps. The corresponding
rate function can be always written in a (simple) closed form. The Gallavotti-Cohen large
deviation principle \cite{M,LS} and the associated fluctuation theorem
can be obtained by projection \cite{BFG3}. Moreover, by contraction one also derives
 a dual variational formula for the rate function of the empirical measure.

Alternatively, the joint large deviations for the empirical measure and flow 
 can be directly derived  by  tilting the underlying Markov chain. 
Indeed, with this approach it has  been firstly  derived  in \cite{WK}
for a Markov chain with two states. Always in the context of discrete state space,
a large deviations principle for  flows and currents
have been discussed in \cite{BMN} in relation to statistical mechanics models.
The general case of countable state space is analyzed in \cite{BFG}, to which we refer for further references.

With respect to this setting, the phononic chain described above lives on a  continuous state space and lacks uniform ergodicity due to the presence of zero as absorbing state. In particular the classical Donsker-Varadhan  conditions 
\cite{DV4, DeS} do not hold. 
Motivated by this model, 
we consider a class of  continuous time Markov chains    which are degenerate in the sense that there exist  states
with infinite holding time, but the corresponding skeletons admit an invariant measure for which the mean jump time is integrable.
For simplicity, we restrict to the case of compact state space.
We prove a large deviation principle for the  empirical measure and flow, with an initial state different from the absorbing ones. 

The rate function is the  continuous version of the one derived in \cite{BFG} for discrete space states.
The presence of  absorbing states is however reflected in the  properties of the rate function. 
Its zero level set is not a singleton and more precisely it contains convex combinations of the invariant measure with the associated flow 
and measures supported by the absorbing states  with zero flow. 
Indeed, with sub-exponential probability, the chain may spend almost all the time in a small neighborhood of the absorbing states. 
Analogous degenerate large deviation asymptotics have been obtained in 
\cite{LMZ,LMZ2} in the context of renewal processes and in \cite{BLT,BT} in
the context of interacting particle systems.

From the large deviation principle for the empirical measure and flow we deduce by contraction the large deviation principle 
for the empirical measure. The corresponding 
rate function can be expressed by the Donsker-Varadhan variational formula, which in this case also admits a not trivial zero level set.
Furthermore, we also obtain a variational expression for the rate function describing the large deviation asymptotics of the empirical flow.

The large deviation upper bound for the empirical measure and flow is proven   
 by perturbing the rates of the
underlying Markov chain. We remark that
this step can be accomplished since  the Radon-Nikodym derivative of the corresponding laws can be expressed in
terms of the empirical measure and flow.
We derive the lower bound by considering first deviations of measures and flows with support bounded away from the absorbing state. 
For this class we can construct perturbed Markov chains with nice ergodic properties, which have these measures and flows as typical behavior. 
We then complete the proof by a density argument.
  
 


\section{Notation and results}
\label{definizioni}

Let $E$ be a compact Polish space, i.e.\ metrizable complete and
separable, endowed with its Borel $\sigma$-algebra.  
The spaces of continuous functions on $E$ and $E\times E$, endowed
with the  uniform norm $\|\cdot\|$, are denoted by 
$C(E)$ and $C(E\times E)$.
We consider a continuous time Markov chain $\xi_t$, $t\in \bb R_+$ on
the state space $E$, defined by transitions rates $c(x,dy)=r(x)
p(x,dy)$, where $r\colon E\to \bb R_+$ and $p$ is a transition kernel on
$E$. Throughout all the paper we assume the transition rates satisfy
the following conditions which in particular imply that $\xi$ is not
explosive and Feller.

\begin{assumption}
  \label{t:1.1}$~$
  \begin{itemize}
  \item[(i)] 
    The function $r\colon E\to \bb R _+$ is continuous.
    We set $E_0:=\{x\in E:\, r(x)=0\}$.
  \item[(ii)] 
    There exists a probability $\lambda$ on $E$, strictly
    positive on open sets, such that $p(x,dy)= p(x,y)\lambda(dy)$ for
    some strictly positive density $p\in C(E\times E)$.
  \item[(iii)] 
    The function $1/r$ is integrable with respect to $\lambda$, i.e., 
    $\lambda\big(1/r\big) <+\infty$.
  \end{itemize}  
  In order to prove the large deviations lower bound we also need the
  following technical condition.
  \begin{itemize}
  \item[(iv)] 
    For $\delta>0$, let $A_\delta:=\{x\in E:\,r(x)<\delta\}$ be the
    (open) level set of $r$.  There exists a sequence $\delta_n\to 0$ such
    that $\sup_n\lambda(A_{2\delta_n}) /
    \lambda(A_{2\delta_n}\!\!\!\setminus\!  A_{\delta_n})<+\infty$.
  \end{itemize}
\end{assumption}

Since $r$ is continuous $E_0$ is closed. Moreover, in view of condition (iii) $\lambda(E_0)=0$.
Assumption (ii) implies that the kernel $p$ is Feller and satisfies
the Doeblin condition.  In view of \cite[Thm.~16.0.2]{MT} the discrete
time Markov chain with kernel $p$ is uniform ergodic.  That is, there
exists a probability $\pi_0$ on $E$ such that $p^n(x,\cdot)$ converges
in total variation to $\pi_0$ uniformly with respect to $x\in E$.
Moreover, since $\lambda$ is strictly positive on open sets, $\pi_0$
enjoys the same property.  In view of items (ii) and (iii), $\pi_0
\big(1/r \big)<+\infty$ and therefore
\begin{equation}
  \label{mm}
  \pi(dx):= \frac 1 { \pi_0 \big(1/r \big)} \, \frac {\pi_0(dx)}{r(x)} 
\end{equation}
defines a probability on $E$. 
As it is simple to check, $\pi$ is an invariant probability for the
continuous time chain $\xi$.

As discussed in the Introduction, the main novelty of this paper is
that we allow the set $E_0$ to be not empty. If this is the case, the points in $E_0$ are  absorbing
states for the chain $\xi$. In particular any probability supported on a subset of $E_0$ is also an invariant
measure and $\xi$ is not uniformly ergodic. Then the
standard conditions for the Donsker-Varadhan theorem, see e.g.,
\cite{DeS,DV4} do not hold. The phononic chain
described by \eqref{Be0} (see \cite{BOS} for the explicit expression of the rates) 
meets the requirements in Assumption~\ref{t:1.1} with $E_0$ the singleton at the point $0$.

We next state the ergodic theorem for the chain $\xi$. 
For $x\in E$ we denote by $\bb P_x$ the distribution of the process
$\xi$ with initial condition $x$.  Observe that $\bb P_x$ is a
probability on the Skorokhod space $D(\bb R_+;E)$ whose 
canonical coordinate will be denoted by $X_t$, $t\in \bb R_+$.
The expectation with respect to  $\bb P_x$ is denoted by  $\bb E_x$.

\begin{theorem}
  \label{t:lln}
  Let $f\in C(E)$ and $x\in E\setminus E_0$. Then 
  \begin{equation*}
    \lim_{T\to +\infty} \frac 1T \int_0^T\!dt\, f(X_t) = \pi(f)
    \qquad \textrm{in $\bb P_x$ probability.}
  \end{equation*}
  Moreover, the convergence is uniform with respect to $x$ in a
  compact subset of $E\setminus E_0$.
\end{theorem}

We denote by $\mc M_1(E)$ the space of probability measures on $E$
endowed with the topology of weak convergence.  Given $T>0$, the
\emph{empirical measure} $\mu_T$ is the continuous map from $D(\bb R_+;E)$
to $\mc M_1(E)$ defined by
\begin{equation}\label{em}
  \mu_T(f)\, (X) := \frac 1T\int_0^T\!dt\, f(X_t), \qquad f\in C(E).
\end{equation}
Theorem~\ref{t:lln} can be then restated as follows.
As $T\to +\infty$ the family $\big\{\bb P_x\circ
\mu_T^{-1}\big\}_{T>0}$ converges to $\delta_{\pi}$ uniformly with
respect to $x$ in a compact subset of $E\setminus E_0$.

To describe the large deviation asymptotic of the empirical measure,
we follow the approach introduced in \cite{BFG} for discrete state
space. Within this scheme, together with the empirical measure it is
also considered the empirical flow which accounts for the number of
jumps between two given states. For this purpose, we let $\mc
M_+(E\times E)$ be the space of finite positive measures on $E\times
E$ equipped with the \emph{bounded weak*} topology.  This is defined
as follows. Let $\mc M(E\times E)$ be the set of finite signed
measure on $E\times E$. The weak* topology on $\mc M(E\times E)$ is
then defined by identifying it with the dual of $C(E\times E)$. For
$Q\in\mc M(E\times E)$ denote by $\|Q\|_\mathrm{TV}$ the total
variation of $Q$ and, given $\ell>0$, let $B_\ell := \big\{ Q\in \mc M
:\, \|Q\|_\mathrm{TV} \le \ell\big\}$ be the closed ball of radius
$\ell$ in $\mc M(E\times E)$.  The bounded weak* topology on $\mc
M(E\times E)$ is then defined by declaring a set $A\subset \mc
M(E\times E)$ open if and only if $A\cap B_\ell$ is open in the weak*
topology of $B_\ell$ for any $\ell>0$. In particular, the bounded
weak* topology is stronger than the weak* topology and, as follows
from the Banach-Alaoglu theorem, for each $\ell>0$ the closed ball
$B_\ell$ is compact with respect to the bounded weak* topology.  The
space $\mc M(E\times E)$ endowed with the bounded weak* topology is a
locally convex, complete, linear topological space, and a completely
regular space, i.e., for each closed set $C \subset \mc M(E\times E)$
and each $Q \in \mc M (E\times E)\setminus C$ there exists a
continuous function $f\colon\mc M(E\times E) \to [0,1]$ such that
$f(Q)=1$ and $f(Q')=0$ for all $Q'\in C$).  We refer to
\cite[\S~2.7]{Me} for the proof of the above statements and for
further details.  We finally regard $\mc M_+(E\times E)$ as a (closed)
subset of $\mc M (E\times E)$ and consider it endowed with the
relative topology and the associated Borel $\sigma$-algebra.

For $T>0$, the \emph{empirical flow} $Q_T$ is the map from $D(\bb R_+;E)$ to
$\mc M_+(E\times E)$ defined $\bb P_x$ a.s., $x\in E$, by
\begin{equation*}
   Q_T(F)\, (X) := \frac 1T\sum_{\substack{t\in [0,T] \,: \\ \:X_{t^-}\neq X_t}} 
    F\big(X_{t^-},X_t\big), 
    \qquad F\in C(E\times E).
\end{equation*}
Observe indeed that the right hand side is well defined because $\bb
P_x$ a.s., $x\in E$, the set of discontinuities of $X_t$ is locally
finite. 
In view of Theorem~\ref{t:lln}, a straightforward martingale
decomposition, see Proposition~\ref{t:lbns} below, yields 
the following law of large numbers for empirical flow.  
Let $Q^\pi(dx,dy):=\pi(dx)c(x,dy)$, then
as $T\to +\infty$ the family $\big\{\bb P_x\circ Q_T^{-1}\big\}_{T>0}$
converges to $\delta_{Q^\pi}$ 
uniformly with respect to $x$ in a compact subset of
$E\setminus E_0$.
 
We regard the pair $(\mu_T, Q_T)$ as a map from $D(\bb
R_+;E)$ to the product space $\ms M :=\mc M_1(E)\times \mc
M_+(E\times E)$ defined $\bb P_x$ a.s., $x\in E$.
Our main result is the large deviation principle for the family
$\big\{\bb P_x\circ (\mu_T, Q_T)^{-1}\big\}_{T>0}$. We start by
defining the rate function.  
Let $\Psi\colon\bb R_+\to \bb R_+$ be the convex function 
$\Psi (a) := a \log a -(a-1)$, in which we understand that
$\Psi(0)=1$. We then define the functional
$I\colon \ms M\to [0,+\infty]$  by
\begin{equation}
  \label{I}
  I(\mu,Q) :=
  \begin{cases}
    \displaystyle \iint\!\mu(dx)c(x,dy) \:
    \Psi\Big(\frac{Q(dx,dy)}{\mu(dx) c(x,dy)}  \Big) 
    & \textrm{ if }\;  Q(\cdot, E)=Q(E,\cdot),
    \\
    \vphantom{\big\{^\{}
    +\infty & \textrm{ otherwise. }
  \end{cases}
\end{equation}
Observe that $I(\mu, Q)<+\infty$ implies that the two marginals of $Q$ 
are equal and $Q(dx,dy)\ll \mu(dx)c(x,dy)$.
Moreover, since the second marginal of $\mu(dx) c(x,dy)$ is
absolutely continuous with respect to $\lambda$, $I(\mu,Q)<+\infty$
also implies $dQ= q\, d\lambda\times d\lambda$, see Lemma~\ref{t:rl}
below. It is thus possible to express $I$ in terms of the density $q$ as
follows.  Given $\mu\in \mc M_1(E)$, decompose it into the
absolutely continuous and singular parts with respect to $\lambda$.
 Namely, $d\mu=\varrho \, d\lambda + d\mu_{\mathrm{s}}$ where
$\varrho$ is a sub-probability density on $E$ and $\mu_{\mathrm{s}}$
is singular with respect to $\lambda$. 
Let $\Phi\colon {\bb R_+^2\to[0, +\infty]}$ be the convex lower
semi-continuous function defined by $\Phi(a,b):=a\log (a/b)-a+b$.  
If $I(\mu, Q)<+\infty$ then the marginals of $Q$ are equal and
\begin{equation}\label{I1}
  I(\mu,Q)=\iint \! \lambda(dx)\lambda(dy) \: 
  \Phi\big(q(x,y),\varrho(x)r(x)p(x,y)\big) +\mu_{\mathrm{s}}(r).
\end{equation}
In particular, if $\mu=\mu_{\mathrm{s}}$ then $I(\mu,Q)<+\infty$ implies $Q=0$.

\begin{theorem}
  \label{t:ldp}
  As $T\to +\infty$ the family $\big\{\bb P_x\circ (\mu_T,
  Q_T)^{-1}\big\}_{T>0}$ satisfies, uniformly with respect to $x$
  bounded away from $E_0$, a large deviation principle with good
  convex rate function $I$. Namely, the functional $I$ has compact
  level sets and for each compact $E_1\subset E\setminus E_0$, each
  closed $C\subset \ms M$, respectively each open $A\subset \ms M$,
  \begin{equation*}
    \begin{split}
      &\varlimsup_{T\to+\infty}\;\sup_{x\in E_1} \;\frac 1T \log \bb
      P_x\big( (\mu_T,Q_T)\in C \big)
      \le - \inf_{(\mu,Q)\in C}  I (\mu,Q), \\
      &\varliminf_{T\to+\infty}\;\inf_{x\in E_1} \;\frac 1T \log \bb
      P_x\big( (\mu_T,Q_T)\in A \big) 
      \ge - \inf_{(\mu,Q)\in A} I(\mu,Q).
    \end{split}
  \end{equation*}
\end{theorem}

By the stationarity condition for $\pi$, the
measure $Q^\pi(dx, dy)=\pi(dx)c(x,dy)$ has equal marginals. We thus
deduce, as must be the case, that $I(\pi, Q^\pi)=0$. On the other
hand, if $E_0$ is not empty, the zero level set of $I$ contains other points
and the law of large numbers stated in Theorem~\ref{t:lln} cannot be
deduced from the large deviation result. More precisely, if the measure $\mu$ is supported on a subset of $E_0$, then
$I(\mu, 0)=0$ and, by convexity, $I$ vanishes on the
segment $\alpha(\pi, Q^\pi)+(1-\alpha)(\mu,0)$,
$\alpha\in[0,1]$.  The representation \eqref{I1} implies  that elements of this form are the only zeros of $I$.

As a corollary of the previous theorem, we deduce the large deviations
asymptotic for the empirical measure. We emphasize that the
corresponding rate function is the standard Donsker-Varadhan
functional.

\begin{corollary}
  \label{t:ldem}
  Let $\hat I \colon \mc M_1(E)\to [0,+\infty]$ be the functional defined by 
  \begin{equation*}
	\hat I (\mu) =\sup_{\phi\in C(E)} \Big\{ - \iint \!\mu(dx)c(x,dy) \, 
    \big[ \exp\{\phi(y)-\phi(x)\} -1 \big]\Big\}.
  \end{equation*}
  As $T\to +\infty$ the family $\big\{\bb P_x\circ \mu_T^{-1}\big\}_{T>0}$
  satisfies, uniformly with respect to $x$ bounded away from $E_0$, a
  large deviation principle with convex rate function $\hat I$.
\end{corollary}

As a further projection of Theorem \ref{t:ldp}, we  obtain a variational expression, that appears to be new, 
of the rate function for the empirical flow.

\begin{corollary}
  \label{t:ldem2}
  Let $\tilde I \colon \mc M_+(E\times E)\to [0,+\infty]$ be the functional defined by 
  \begin{equation*}
    \tilde I (Q) =
     \begin{cases}
\displaystyle\sup_{\alpha\in (-r_{\mathrm m},\,+\infty)} \Big\{ \iint \!Q(dx,dy) \, 
    \log \Big[ \frac{Q(dx,dy)}{Q(dx, E)c(x,dy)}(r(x)+\alpha)   \Big]\; -\alpha\Big\}\vspace{0.2cm}\\
\qquad\qquad \qquad  \qquad \mathrm{if }\;Q(\cdot, E)=Q(E, \cdot)\vspace{0.2cm}\\
+\infty \qquad\qquad \qquad\mathrm{otherwise},
 \end{cases}
\end{equation*}
where $r_{\mathrm m}:=\min r$.
As $T\to +\infty$ the family $\big\{\bb P_x\circ Q_T^{-1}\big\}_{T>0}$
  satisfies, uniformly with respect to $x$ bounded away from $E_0$, a
  large deviation principle with convex rate function $\tilde I$.
\end{corollary}

\section{Law of large numbers}
\label{sec3}
We denote with $\{Z_i\}_{i\geq 0}$ the skeleton of the process $\xi$,
namely the sequence of the visited states, 
and with $\{\tau_i\}_{i\geq 0}$ the collection of the holding times.
The skeleton $\{Z_i\}_{i\geq 0}$ is a Markov chain with transition
probability $p(x, dy)$.  
Conditioned to the skeleton $\{Z_i\}_{i\geq 0}$, 
$\{\tau_i\}_{i\geq 0}$ are independent, exponentially distributed
random variables with parameters $r(Z_i)$.  
In particular they have the same law as 
$\{ r(Z_i)^{-1}e_i\}_{i\geq 0}$, where $\{e_i\}_{i\geq 0}$ are i.i.d.\ 
exponential random variables with parameter $1$.

We denote with $T_n$, $n\geq0$, the jump times
$T_n:=\sum_{i=0}^{n-1}\tau_i$ for $n\ge 1$ and $T_0=0$.  
We then define the clock process $\mc T (t):=T_{\lfloor t \rfloor}$, where
$\lfloor\cdot\rfloor$ denotes the integer part.  
The inverse function $n(t):=\inf\{n:\;T_n\geq t\}$ gives the number of
jumps up to time $t$.
By definition, the following inequality holds
\begin{equation}
  \label{ineq:Tn}
  T_{n(t)-1}<t\leq T_{n(t)},
\end{equation}
where we take $T_{-1}=0$.

\begin{proposition}
  \label{prop:lln} 
  Let $\pi_0\in\mc M_1(E)$ be the unique invariant measure of the
  chain $\{Z_i\}$. Then for each $f\in C(E)$ and $Z_0\in
  E\setminus{x_0}$
  \begin{equation}
    \label{eq:lln}
    \lim_{n\to\infty}\frac{1}{n}\sum_{i=0}^{n-1}f(Z_i)\tau_i= \pi_0(f/r)
    \qquad \textrm{in probability.}
  \end{equation}  
  Moreover, the convergence is uniform with respect to
  $Z_0$ in a compact subset of $E\setminus E_0$.
\end{proposition}

Postponing the proof above statement, we first show that it implies 
the law of large numbers for the continuous time chain $\xi$.

\begin{proof}[Proof of Theorem \ref{t:lln}]
  Recalling the definition of the empirical measure \eqref{em},
  \begin{equation*}
    \mu_T(f)=\frac 1{T}\sum_{i=0}^{n(T)}f(Z_i)\tau_i.
  \end{equation*}
  For each $f\in C(E)$, $\epsilon>0$, and $\sigma>0$
  \begin{equation}\begin{split}
      \label{eq:lln0}
      & \bb P_x \Big( \big|\mu_T(f) - \pi(f)\big|>\epsilon\Big)
      \leq  
      \bb P_x\Big( \Big|\frac 1{T} n(T)
      - \frac 1{\pi_0(1/r)}\Big|>\sigma\Big)
      \\ 
      &\qquad \qquad 
      + \bb P \Big( \big|\mu_T(f) - \pi(f)\big|>\epsilon,
      \; \Big|\frac 1{T} n(T)- \frac 1{\pi_0(1/r)}\Big|\leq\sigma \Big).
    \end{split}
  \end{equation}

  From \eqref{eq:lln} with $f=1$, we deduce that the sequence
  $\{T_n/n\}_{n\geq 1}$ converges in probability to $\pi_0(1/r)$ and
  therefore, in view of \eqref{ineq:Tn}, the family of random
  variables $\{n(T)/T\}_{T> 0}$ converges in probability to
  $\big(\pi_0(1/r)\big)^{-1}$. This implies that the first term on the
  right hand side of \eqref{eq:lln0} vanishes as $T\to\infty$.

  On the other hand, on the event $\big\{\big|\frac 1{T} n(T)- \frac
  1{\pi_0(1/r)}\big|\leq\sigma\big\}$,
  \begin{equation*}
    \Big\vert \mu_T(f)-\frac1{T}\sum_{i=0}^{\alpha(T)}f(Z_i)\tau_i\Big\vert
    \leq \|f\| \, \frac{1}{T}
    \sum_{i=\alpha(T)-\lfloor\sigma T\rfloor -1}^{\alpha(T)+\lfloor 
      \sigma T\rfloor +1}\tau_i,
  \end{equation*}
  where $\alpha(t)=\lfloor t/\pi_0(1/r)\rfloor$.  In view of condition
  (ii) in Assumption \ref{t:1.1}, for each $i\geq 1$ and $Z_0\in E$
  \begin{equation*}
    \bb E \big( \tau_i\big)=\int p^{i-1}(Z_0, dx)\int p(x, dy)\frac 1 {r(y)}
    \leq C\int p^{i-1}(Z_0, dx)\int \lambda(dy)\frac 1 {r(y)},
  \end{equation*}
  where $C=\max p(x,y)$. By condition (iii) in Assumption \ref{t:1.1},
  we thus get
  \begin{equation}
    \label{etaui}
    \sup_{i\geq 1}\: \bb E \big( \tau_i\big)<+ \infty.
  \end{equation}
  Hence, by Chebychev inequality, the second term on the right hand
  side of \eqref{eq:lln0} vanishes as $\sigma\to 0$ uniformly in $T$.
\end{proof}

\begin{proof}[Proof of Proposition \ref{prop:lln}]
  Recalling that, conditionally on $\{Z_i\}$,  $\{\tau_i\}$ have the same law as
  $\{r(Z_i)^{-1}e_i\}$, we set $S_n(f)=\frac
  1{n}\sum_{i=0}^{n-1}f(Z_i)r(Z_i)^{-1}e_i$ and define
  \begin{equation*}
    u^<_{n,i}:=\frac1{r(Z_i)}\id_{\{r(Z_i)^{-1}\leq n^{1/4}\}},
    \qquad
    u^>_{n,i}:=\frac1{r(Z_i)}\id_{\{r(Z_i)^{-1}> n^{1/4}\}}.
  \end{equation*}
  We decompose accordingly
  \begin{equation}\label{Sn}
    S_n(f)=\frac1 n f(Z_0)r(Z_0)^{-1}e_0+S_n^<(f)+S_n^>(f),
  \end{equation}
  where
  \begin{equation*}
    \begin{split}
      S_n^< (f)& := \frac{1}{n}\sum_{i=1}^{n-1}f(Z_i)u^<_{n,i}e_i,\\
      S_n^> (f)& := \frac{1}{n}\sum_{i=1}^{n-1}f(Z_i)u^>_{n,i}e_i.
    \end{split}
  \end{equation*}

  Trivially, the first term on the r.h.s.\ of \eqref{Sn} vanishes as
  $n\to\infty$, uniformly with respect to $Z_0$ in a compact subset of
  $E\setminus E_0$.  Let us next show that $S_n^>(f) $ converges
  to zero in $L^1$. We have
  \begin{equation*}
    \bb E\big( |S_n^>(f)|\big) \leq \frac 1 {n}\, \|f\|
    \, \sum_{i=1}^{n-1}\bb E\big(u^>_{n,i}\big), 
  \end{equation*}
  where, by Chapman-Kolmogorov,
  \begin{equation*}
    \begin{split}
      &\frac 1{n} \, \|f\| \, \sum_{i=1}^{n-1}
      \int \! p^{i-1}(Z_0, dx)\int\!  p(x,dy)\;
      \frac 1{r(y)}\id_{\{r(y)^{-1}> n^{1/4}\}}
      \\ 
      &\qquad\quad \leq \|f\|\, \sup_{x\in E}
      \int \! p(x, dy)\, 
      \frac 1{r(y)}\id_{\{r(y)^{-1}> n^{1/4}\}}, 
    \end{split}
  \end{equation*}
  which vanishes as $n\to\infty$ by conditions (ii) and (iii) in
  Assumption \ref{t:1.1}.

  We next show that $S_n^<(f)$ converges to $\pi_0(f/r)$ in $L^2$.
  We have
  \begin{equation}
    \label{S<1}
    \begin{split}
      & \bb E\big(S_n^<(f)-\pi_0(f/r) \big)^2 \leq 
      \frac 1 {n^2}\sum_{i=1}^{n-1} 
      \bb E\Big( f(Z_i)u^<_{n,i}e_i -\pi_0(f/r)\Big)^2
      \\ 
      &\qquad  +\frac 2 {n^2}\sum_{1\leq i<j\leq n-1} 
      \bb E\Big( \big[f(Z_i)u^<_{n,i}e_i -\pi_0(f/r)\big]
      \big[f(Z_j)u^<_{n,j}e_j -\pi_0(f/r)\big] \Big).
    \end{split}
  \end{equation}
  By definition of $u^<_{n,i}$, 
  $\bb E \big(f(Z_i)u^<_{n,i}e_i\big)^2\leq \|f\|^2 n^{1/2}$. 
  Therefore the first sum on the right hand side of \eqref{S<1}
  vanishes as $n\to\infty$.
  In order to estimate the second sum on the r.h.s.\ of \eqref{S<1}, we
  observe that for each $1\leq i<j\le n-1$
  \begin{equation*}
    \begin{split}
      & \Big|\bb E
      \Big( f(Z_i)\big[u^<_{n,i}-r(Z_i)^{-1}\big]
      \, f(Z_j) \big[u^<_{n,j}-r(Z_j)^{-1}\big] \Big)
      \\  
      &\qquad \quad 
      \leq \|f\|^2 \, \Big(\sup_{x\in E}
      \int \!p(x, dy)\frac 1{r(y)}\id_{\{r(y)^{-1}> n^{1/4}\}}\Big)^2,
    \end{split}
  \end{equation*}
  which vanishes as $n\to\infty$ by conditions (ii) and (iii) in Assumption
  \ref{t:1.1}.  Therefore we can replace $u^<_{n,i}$ by $r^{-1}(X_i)$
  in the second term of the r.h.s. of \eqref{S<1}.
  By the same computations presented above, we get that 
  each term in this modified sum is uniformly bounded, that is
  \begin{equation}\label{upES}
    \Big| \bb E\Big( \big[ f(Z_i)r(Z_i)^{-1}e_i -\pi_0(f/r)\big]
    \big[f(Z_j) r(Z_j)^{-1}e_j -\pi_0(f/r)\big] \Big) \Big| 
    \le C^2\, \|f\|^2,
  \end{equation}
  where $C:=\sup_{x\in E} \int\! p(x,dy)\, 1/r(y)$.  
  Given $m<n-1$, we now split the sum
  \begin{equation*}
    \sum_{1\leq i<j\leq n-1}=\sum_{\substack{1\leq i<m\\i< j\leq n-1}}
    +\sum_{\substack{m\leq i<j\leq n-1\\ j-i\leq m}}
    +\sum_{\substack{m\leq i<j\leq n-1\\ j-i> m}}.
  \end{equation*}
  By \eqref{upES}, the first and the second sum give a contribution of
  order $m/n$, then it remains to estimate the last sum.
  Since $\pi_0(\cdot) =\int\!\pi_0(dx) p(x,\cdot)$, for $\ell\geq 0$
  we can write
  \begin{equation}
    \label{S<2}
    \begin{split}
      &\Big|\int\! p^{\ell+1}(x,dy) \, \frac {f(y)}{r(y)}
      -\pi_0(f/r)\Big|
            = \Big| \int \!\big(p^{\ell}(x,dz)-\pi_0(dz)\big)
      \int\! p(z, dy) \, \frac {f(y)} {r(y)}\Big|
      \\
      &\qquad\qquad \leq C\, \| f\|\, 
        \, \sup_{x\in E}\|p^{\ell}(x,\cdot)-\pi_0(\cdot)\|_{\mathrm{TV}}.
    \end{split}
  \end{equation}
  Therefore
  \begin{equation*}
    \begin{split}
      &\frac 1 {n^2}\sum_{\substack{m\leq i<j\leq n-1\\ j-i> m}} 
      \Big| \bb E\Big( 
      \big[f(Z_i)r(Z_i)^{-1} e_i -\pi_0(f/r)\big]
      \big[f(Z_j) r(Z_j)^{-1}e_j -\pi_0(f/r)\big] 
      \Big)\Big| \\
      &\qquad \quad
      \leq C^2\, \| f\|^2\,
      \sup_{\ell\geq m} \Big(
      \sup_{x\in E} \|p^{\ell}(x,\cdot)-\pi_0(\cdot) \|_\mathrm{TV} \Big)^2.
    \end{split}
  \end{equation*}
  Since condition (ii) in Assumption~\ref{t:1.1} implies the Doeblin
  condition for $p$, the uniform ergodicity of the chain, see e.g., \cite[Thm.
  16.0.2]{MT}, implies that the r.h.s.\ above vanishes as
  $m\to\infty$.
  We then conclude the proof taking the limit $n\to \infty$ and then
  $m\to\infty$.
\end{proof}

\section{Large deviations upper bound}
\label{s:ub}

We denote the marginals of $Q\in \mc M_+(E\times E)$ by $Q^{(1)}$ and
$Q^{(2)}$. For $F\in C(E\times E)$ we let $r^F\colon E\to \bb R_+$ be
the continuous function defined by 
\begin{equation}
  \label{def:rf}
  r^F(x):=\int c(x, dy)\,e^{F(x,y)};
\end{equation}
observing that for $F=0$ we get $r^0=r$.
Given $\phi\in C(E)$ and $F\in C(E\times E)$ let $I_{\phi,F}\colon \ms
M\to \bb R$ be the continuous affine map defined by
\begin{equation}
  \label{IF}
  I_{\phi,F}(\mu,Q) :=
  Q^{(1)}(\phi)-Q^{(2)}(\phi)+  Q(F) - \mu\big(r^F-r\big).
\end{equation}
In this section we first prove, by an exponential tilt of the
underlying probability, the large deviation upper bound with rate
function $\sup_{\phi,F}I_{\phi,F}$. 
As in \cite{BFG}, this step can be easily accomplished since we are
considering the joint deviations of the empirical measure and flow. We
then show that the rate function thus obtained coincides with
\eqref{I}.  We remark that the upper bound estimate holds uniformly
with respect to all initial conditions in $E$.

\begin{proposition}
  \label{t:ub}
  As $T\to +\infty$ the family $\big\{\bb P_x\circ (\mu_T,
  Q_T)^{-1}\big\}_{T>0}$ satisfies, uniformly with respect to $x\in
  E$, a large deviation upper bound with lower semi-continuous convex
  rate function $\sup_{\phi,F}I_{\phi,F}$. Namely, for each closed
  $C\subset \ms M$
  \begin{equation*}
    \varlimsup_{T\to+\infty}\sup_{x\in E}
    \frac 1T \log \bb P_x\big( (\mu_T,Q_T)\in C \big)
    \le - \inf_{(\mu,Q)\in C} \sup_{\phi,F}I_{\phi,F}(\mu,Q)
  \end{equation*}
  where the supremum is carried out over all 
  $(\phi,F)\in C(E)\times C(E\times E)$.
\end{proposition}

We start by proving the exponential tightness,  that is there
exists a sequence $\{\mc K_\ell\}_{\ell\in \bb N}$ of compacts in 
$\ms M$ such that
\begin{equation*}
  \lim_{\ell\to\infty} \: \varlimsup_{T\to\infty} \: 
  \sup_{x\in E}\frac 1 T\, 
  \log \bb P_x\big( (\mu_T,Q_T) \notin \mc K_\ell\big)=-\infty.
\end{equation*}
Recall $\ms M=\mc M_1(E)\times \mc M_+(E\times E)$. Since $\mc M_1(E)$
is compact with respect to the topology of weak convergence and $\mc
M_+(E\times E)$ is endowed with the bounded weak* topology, the
previous bound follows from the exponential tightness of the sequence
of positive random variables $\{Q_T(1)\}_{T>0}$, which count the total
number of jumps per unit of time.

\begin{lemma}
  \label{t:et}
  Let $a_\ell\to +\infty$. Then
  \begin{equation*}
    \lim_{\ell\to\infty} \:\varlimsup_{T\to\infty} \:\sup_{x\in E}
    \frac 1T\,\log \bb P_x\big( Q_T(1)  >a_\ell \big) =-\infty.
  \end{equation*}
\end{lemma}

\begin{proof}
  Given $F\in C(E\times E)$, let $\bb M^F$ be the process defined by
  \begin{equation}\label{M}
    \bb M^F_t =\exp\big\{t\big[Q_t(F)-\mu_t(r^F -r) \big  ]\big\},
    \qquad t\in \bb R_+.
  \end{equation}
  By standard Markov chain computations, see e.g., \cite[\S VI.2]{Br}, $\bb
  M^F$ is a mean one positive $\bb P_x$ martingale, $x\in E$.
  By choosing $F(x,y)=\gamma >0$, $(x,y)\in E\times E$, for $a>0$,
  $T>0$ we then write
  \begin{equation*}
    \begin{split}
      \bb P_x\big(Q_T(1) >a \big) 
      &=  
      \bb E_x\big( e^{-T\{ \gamma Q_T(1) -\mu_T(r^\gamma-r)\}} \, 
      \bb M^F_T \, \id_{\{Q_T(1) >a\}}  \big)
      \\
      & 
      \leq e^{-T\gamma a } \, e^{T\|r\|(e^\gamma-1)}
      \, \bb E_x\big(\bb M^F_T \big) 
      = e^{-T\gamma a } \, e^{T\|r\|(e^\gamma-1)}.
    \end{split}
  \end{equation*}
  The statement follows.
\end{proof}

\begin{lemma}
  \label{t:pub}
  For each $(\phi,F)\in C(E)\times C(E\times E)$ and each measurable
  $\mc B\subset \ms M$,
  \begin{equation*}
    \varlimsup_{T\to \infty}\; \sup_{x\in E}\,\frac 1{T}
    \log \bb  P_{x} \Big( (\mu_{T},Q_{T}) \in \mc B \Big)
    \le -\inf_{(\mu,Q)\in \mc B}  I_{\phi,F} (\mu,Q).
  \end{equation*}
\end{lemma}

\begin{proof}
  Fix $x\in E$ and observe that the following path-wise continuity
  equation holds $\bb P_x$ a.s.,
  \begin{equation}
    \label{pce}
    \begin{split}
      0 & = \phi(X_T)-\phi(X_0) -  
      \sum_{t\in[0,T]}\big[\phi(X_{t})-\phi(X_t^-)\big]
      \\
      &=\phi(X_T)-\phi(X_0)- T \big[Q_T^{(2)}(\phi)-Q_T^{(1)}(\phi)\big].
    \end{split}
  \end{equation}
  In view of \eqref{IF} and \eqref{pce}, recalling the martingale
  introduced in \eqref{M}, for each $T>0$
  \begin{equation*}
    \begin{split}
      &\bb P_x \big( (\mu_T,Q_T) \in \mc B \big)
      \\
      &\qquad
      = \bb E_x \Big(
      \exp\big\{ -T \, I_{\phi,F} (\mu_T,Q_T) -
      \big[ \phi(X_T)-\phi(x) \big] \big\}
      \: \bb M_T^F \: \id_{\mc B}(\mu_T,Q_T) \Big)
      \\
      &\qquad \le
      \sup_{(\mu,Q)\in\mc B} e^{- T \, I_{\phi,F} (\mu,Q) }
      \; \bb E_x \Big(
      \exp\big\{- \big[\phi(X_T)-\phi(x)\big] \big\}
      \: \bb M_T^F \: \id_{\mc B}(\mu_T,Q_T) \Big)\\
      &\qquad \leq \sup_{(\mu,Q)\in\mc B} 
       e^{- T \, I_{\phi,F} (\mu,Q) }e^{2\|\phi\|},
    \end{split}
  \end{equation*}
   where in the last step we used $\bb E_x\big(\bb M_T^F\big)=1$.  The
   statement follows.  
\end{proof}

\begin{proof}[Proof of Proposition \ref{t:ub}]
  In view of the exponential tightness proven in Lemma~\ref{t:et}, it
  is enough to prove the upper bound for compacts.
  For each compact $\mc K\subset \ms M$, by  Lemma~\ref{t:pub} and the min-max lemma in \cite[App.~2,
  Lemma~3.3]{KL}
  \begin{equation*}
    \varlimsup_{T\to+\infty}\,\sup_{x\in E}
    \frac 1T \log \bb  P_{x} \Big( (\mu_T,Q_T) \in \mc K \Big)
    \le -\inf_{(\mu,Q)\in \mc K} \; \sup_{\phi,F} \; I_{\phi,F}(\mu,Q).
  \end{equation*}
 Finally,  as the map $(\mu,Q)\mapsto I_{\phi,F}(\mu,Q)$ is continuous and
  affine, the functional $\sup_{\phi,F} \; I_{\phi,F}$ is lower
  semi-continuous and convex.
\end{proof}

Recalling that the functional $I$ is defined in \eqref{I}, we show
that it coincides with $\sup_{\phi,F} \; I_{\phi,F}$.

\begin{lemma}
  \label{t:rl}
  For each $(\mu,Q)\in\ms M$,
  \begin{equation}\label{I=I}
    I(\mu,Q)= \sup_{\phi,F} I_{\phi,F}(\mu,Q).
  \end{equation}
  In particular, $I$ is lower semicontinuous and convex.  Moreover, if
  $I(\mu,Q) <+\infty$ then $Q\ll \lambda\times\lambda$ and \eqref{I1}
  holds.
\end{lemma}

\begin{proof}
  Clearly, $\sup_\phi \big\{ Q^{(1)}(\phi)-Q^{(2)}(\phi) \big\}
  <+\infty$ if and only if $ Q^{(1)}=Q^{(2)}$. For $\mu\in \mc M_1(E)$
  we denote by $Q^\mu\in \mc M_+(E\times E)$ the measure $Q^\mu(dx,dy)
  := \mu(dx)c(x,dy)$ and set $\Lambda(\mu,Q):=\sup_{F}\big\{ Q(F) -
  Q^\mu\big(e^F-1\big) \big\}$.  Recalling \eqref{def:rf} and
  \eqref{IF}, the proof of \eqref{I=I} is achieved once we show that if
  $Q^{(1)}=Q^{(2)}$ then $\Lambda(\mu,Q)=I(\mu,Q)$.

  For $Q$ with equal marginals we next prove that $\Lambda(\mu,Q)\le
  I(\mu,Q)$. We can assume $I(\mu,Q)<+\infty$ so that $Q\ll
  Q^\mu$. Then 
  \begin{equation*}
    Q(F) - Q^\mu\big(e^F-1\big)
    = \iint \!\! dQ^{\mu} \, 
    \Big\{ \frac{dQ}{dQ^{\mu}}\, F - \big( e^{F} -1 \big) \Big\}.
  \end{equation*}
  Since $\Psi(a)=\sup_{\lambda\in\bb R}\big\{ \lambda a -
  \big(e^\lambda-1\big)\big\}$, $a\in \bb R_+$, we complete this step
  by taking the supremum over $F$.

  To obtain the converse inequality, we first prove that if
  $\Lambda(\mu,Q) < +\infty$ then $Q\ll Q^\mu$.  Let $\tilde B$ be a
  Borel set in $E\times E$ such that $Q^\mu(\tilde B)=0$, we show that
  also $Q(\tilde B)=0$.  By regularity of the measure $Q^\mu$ there
  exists a sequence of open sets $A_n\supset \tilde B$ in $E\times E$
  such that $\lim_n Q^\mu(A_n ) =Q^\mu(\tilde B)=0$.  By approximating
  indicator of open sets with continuous functions we can take as test
  function $F=\gamma \id_{A_n}$, $\gamma>0$, and deduce
  \begin{equation*}
    \gamma \, Q(B) \le \gamma \, Q(A_n) \le 
    \Lambda (\mu,Q) +  \big(e^\gamma-1\big) \, Q^\mu(A_n).  
  \end{equation*}
  We conclude by taking first the limit as $n\to \infty$ and then
  $\gamma\to \infty$.
  To prove $\Lambda(\mu,Q)\ge I(\mu,Q)$ (for $Q$ with equal marginals)
  we can assume $\Lambda(\mu,Q)<+\infty$ so that $Q\ll Q^\mu$.
  Pick an array of continuous functions $\{F_{k,n}\}$
  equibounded in $n$ such that $\{F_{k,n}\}_{n\ge 0}$ converges to 
  $\log\big[\big(dQ/dQ^{\mu} \wedge k \big)\vee 1/k\big]$ in 
  $L^1(E\times E, dQ^\mu)$. Then
  \begin{equation*}
    \begin{split}
      \Lambda(\mu,Q) \ge \iint \!\!d Q^\mu
      \,\frac{dQ}{dQ^\mu}\log\Big[\Big( \frac{dQ}{dQ^{\mu}} \wedge k
      \Big)\vee \frac 1 k \Big] -\iint \!\!d Q^{\mu}\,
      \Big\{\Big[\Big(\,\frac{dQ}{dQ^{\mu}} \wedge k \Big)\vee \frac 1
      k\Big] \,-1\Big\}.
    \end{split}
  \end{equation*}
  By monotone convergence, we conclude taking the limit $k\to\infty$.

  To prove the last statement of the lemma, we decompose the measure
  $\mu$ 
into its absolutely continuous and singular parts  with respect to $\lambda$, i.e.\ 
$\mu=\mu_\mathrm{ac}+\mu_\mathrm{s}$.  Accordingly, there
  exists a Borel set $B\subset E$ such that $\mu_\mathrm{s}(B)
  =\mu_\mathrm{s}(E)$ and $\lambda(B)=0$.  Since $Q^\mu(dx,dy)=
  \mu(dx)r(x) p(x,y) \lambda(dy)$, it holds $Q^\mu(E\times B)=0$. As
  $Q\ll Q^\mu$ and $Q^{(1)}= Q^{(2)}$, this implies $Q(E\times
  B)=Q(B\times E)=0$.  Since the restriction of $Q$ to $(E\setminus B)
  \times E$ is absolutely continuous with respect to
  $Q^{\mu_\mathrm{ac}}\ll \lambda\times \lambda$, then $Q\ll
  \lambda\times \lambda$. Straightforward manipulations now yield
  \eqref{I1}.
\end{proof}

The following estimate will be used in the proof of the lower bound.

\begin{lemma}
  \label{t:1r}
  Let $(\mu,Q)\in \ms M$ be such that  $I(\mu,Q)<+\infty$. Then
  \begin{equation*}
    \iint \! Q(dx,dy) \, \log \frac{1}{r(x)p(x,y)}  <+\infty. 
  \end{equation*}
\end{lemma}

\begin{proof}
  For $k>0$, choose as test function in the variational formula
  \eqref{I=I} the function $(x,y)\mapsto \log \big(k\wedge
  1/r(x) p(x,y)\big)$. We deduce
  \begin{equation*}
    \begin{split}
      & \iint \!\! Q(dx,dy) \, \log\Big( k \wedge\frac{1}{r(x)p(x,y)} \Big)
      \\
      &\qquad 
      \le I(\mu,Q) + \iint \!\! \mu(dx)c(x,dy)\,
      \Big( k \wedge\frac{1}{r(x) p(x,y)} -1\Big)
            \le I(\mu,Q) +1
    \end{split}
  \end{equation*}
  where we used that $c(x,dy)=r(x) p(x,y) \lambda(dy)$.
  By taking the limit $k\to \infty$ we conclude the proof.
\end{proof}

\section{Large deviations lower bound}

We state a general result concerning the large deviation lower bound
in which we denote by $\Ent(\tilde P\vert P)$ the relative entropy of
the probability $\tilde P$ with respect to $P$.

\begin{lemma}
  \label{t:rlb}
  Let $\{P_n^\alpha,\,\alpha\in A\}_{n\in \bb N}$ be a sequence of family of probability measures on a completely
  regular topological space $\mc X$. Assume that for each $z\in\mc X$
  there exists a sequence of family of probability measures $\{\tilde{P}^\alpha_n(z)\}$
  weakly convergent to $\delta_z$ uniformly with respect to $\alpha\in A$ and such that
  \begin{equation}
    \label{entb}
    \varlimsup_{n\to\infty} \sup_{\alpha\in A}\frac 1n \Ent\big(\tilde{P}_n^\alpha(z) \big| P_n^\alpha\big)
    \le J(z)
  \end{equation}
  for some $J\colon \mc X\to [0,+\infty]$. Then the sequence of family $\{P_n^\alpha,\,\alpha\in A\}_{n\in \bb N}$
  satisfies uniformly with respect to $\alpha\in A$ the large deviation lower bound with rate function given
  by $\sce J$, the lower semi-continuous envelope of $J$, i.e.,
  \begin{equation*}
    (\sce J) \, (z) := \sup_{U \in\mc N_z} \; \inf_{w\in U} \; J(w)
  \end{equation*}
  where $\mc N_z$ denotes the collection of the open neighborhoods of $z$.
\end{lemma}

This lemma is proven in \cite[Prop.~4.1]{Je} in a
Polish space setting without the dependence on the parameter $\alpha$. The proof extends 
to the present setting. Note that in the our application we shall work only with sequences so that one can avoid 
the details of the general topological  setting.

Our strategy to prove the large deviations lower bound is the
following. We first prove the lower bound for a nice subset of $\ms
M$. In view on Lemma \ref{t:rlb} we then recover the full lower bound
by a suitable density argument.  More precisely, we let
\begin{equation}
  \label{co}
  \begin{split}
  \ms M_0:=\Big\{& (\mu,Q)\in\ms M \,:\: 
  K:= \supp(\mu)\subset E\setminus E_0,\, 
  \supp(Q)= K \times K,\, 
  \\
  &d\mu =\varrho \, d\lambda 
  \textrm{ with $\varrho$ 
    continuous and $\varrho>0$ on $K$},\, 
  \\
  & dQ =q \,d\lambda\times d\lambda 
  \textrm{ with $q$ 
    continuous and $q>0$ on $K\times K$} \Big\}.     
  \end{split}
\end{equation}
We shall prove the entropy bound \eqref{entb} with $J$
given by the restriction of $I$, as defined in \eqref{I}, to $\ms M_0$, that is
\begin{equation}
  \label{Jrf}
  J(\mu,Q):=
  \begin{cases}
    I(\mu,Q)  & \textrm{ if } (\mu,Q)\in\mc M_0,\\
    +\infty   & \textrm{ otherwise. }
  \end{cases}
\end{equation}
Then we complete the proof of the lower bound  by
showing that the lower semi-continuous envelope of $J$ coincides with
$I$.

\begin{proposition}
  \label{t:lbns}
  Let $(\mu, Q)\in \ms M_0$ and $K:=\supp(\mu)$. Then there exists a
  Markov family $\tilde{\bb P}_x$, $x\in K$, such that $\tilde{\bb
  P}_x\circ (\mu_T, Q_T)^{-1}\to \delta_{(\mu,Q)}$ uniformly with
  respect to $x\in K$ and
  \begin{equation}
    \label{eb}
    \varlimsup_{T\to\infty}\,\sup_{x\in K}\,\frac 1 T\,\Ent
    \big(\tilde {\bb P}_{x,\,[0,T]}\big\vert \bb P_{x,\,[0,T] }\big) 
    \leq I(\mu, Q), 
  \end{equation}
  where $\bb P_{x,\,[0,T] }$ denotes the restriction of $\bb P_x$ to
  $D([0,T],E)$.
\end{proposition}

\begin{proof}
  We can assume $I(\mu,Q)<+\infty$, so that $Q$ has equal marginals.
  Let $d\mu=\varrho \,d\lambda$, $dQ=q \,d\lambda\times d\lambda$, and let
  $\tilde c$ be the transition rates on $K$ defined by $\tilde c
  (x,dy)=q(x,y)/\varrho(x)\,\lambda (dy)$.  We denote by $\tilde{\bb
    P}_x$ the law of the chain with rates $\tilde c$ starting from
  $x$. Since $Q$ has equal marginals, then it is easy to check that
  $\varrho \, d\lambda$ is an invariant measure of the chain.  Moreover,
  the chain is Feller and satisfies $\tilde c(x, dy)\geq c_0 \lambda(dy)$, 
  $x\in K$, with $c_0:=\min_{K\times K} q(x,y)/\varrho(x) >0$. 
  Then, by the arguments of Section \ref{sec3}, $\mu_T$ converges to
  $\varrho\,d\lambda$ in $\tilde{\bb P}_x$ probability, uniformly with
  respect to $x\in K$. In order to prove the law of large numbers for
  the empirical flow $Q_T$, we use the following semi-martingale
  decomposition. For each $F\in C(E\times E)$
  \begin{equation*}
    t\, Q_t(F)=\int_0^t \!ds \int\!  \tilde c(X_s, dy) \, F(X_s,y) +M_t(F),
  \end{equation*}
  where the $\tilde{\bb P}_x$ martingale $M(F)$ has predictable
  quadratic variation 
  \begin{equation*}
    \langle M(F)\rangle_t=\int_0^t \!ds \int\!\tilde c(X_s, dy)\, F(X_s,y)^2. 
  \end{equation*}
  Since $\tilde c(x,dy)\leq C \lambda(dy) $, then $\langle
  M(F)\rangle_t\leq C \,t\, \lambda(F^2)$.  Therefore, as the map
  $K\ni x\mapsto \int\!\tilde c(x,dy) \,F(x,y)$ is continuous, the law of
  large numbers of the empirical measure $\mu_T$ implies
  \begin{equation*}
    \lim_{T\to +\infty} Q_T(F) = 
    \iint \!\mu(dx) \tilde c(x,dy) \, F(x,y) =Q(F),
    \quad \textrm{ in $\tilde{\bb P}_x$ probability},
  \end{equation*}
  uniformly with respect to $x\in K$. Since by Lemma~\ref{t:et} the
  family $\{Q_T\}_{T>0}$ is tight, this implies the law of large
  numbers $\tilde{\bb P}_x\circ (\mu_T, Q_T)^{-1}\to \delta_{(\mu,Q)}$
  uniformly with respect to $x\in K$. 
  Observe that this argument also shows that for each $F$ the family
  of random variables $\{Q_T(F)\}_{T>0}$ converges to $Q(F)$ in $L^2$
  with respect to $\tilde{\bb P}_x$, uniformly in $x\in K$,

  Set $ F^*(x,y):=\log [q(x,y)/\varrho(x)r(x) p(x,y)]$, $(x,y)\in
  K\times K$. Then, by an explicit computation of the Radon-Nikodym
  derivative, see e.g., \cite[\S VI.2]{Br},
  \begin{equation*}
    \frac 1 T\,\Ent \big(\tilde {\bb P}_{x,\,[0,T]}\big\vert \bb
    P_{x,\,[0,T] }\big)
    = \frac 1 T\,\tilde{\bb E}_x 
    \Big( \log \frac{d\tilde {\bb P}_{x,\,[0,T]}}{d\bb
    P_{x,\,[0,T] }}\Big)
    =\tilde{\bb E}_x\Big(Q_T( F^*)-\mu_T(r^{F^*}-r)  \Big).
  \end{equation*}
  Recalling the representation \eqref{I1} for $I$, the law of
  large numbers just proven yields
  \begin{equation*}
    \lim_{T\to\infty} 
    \frac 1 T
    \Ent \big(\tilde {\bb P}_{x,\,[0,T]}\big\vert \bb P_{x,\,[0,T] }\big) 
    = Q(F^*)- \mu(r^{F^*}-r)=I(\mu,Q),
  \end{equation*}
  uniformly with respect to $x\in K$.
\end{proof}

We next show that the lower semi-continuous envelope of $J$, as defined
in \eqref{Jrf}, coincides with $I$.  A set $\ms C \subset \ms M$ is
called $I$-dense in $\ms D \subset\ms M$ if and only if for each $(\mu,Q)\in \ms
D$ such that $I(\mu,Q)<+\infty$ there exists a net 
$\{(\mu_\alpha,Q_\alpha)\} \subset \ms C$ such that
$(\mu_\alpha,Q_\alpha)\to (\mu,Q)$ and $\lim_\alpha
I(\mu_\alpha,Q_\alpha) = I(\mu,Q)$.  We remark that by the lower
semi-continuity of $I$, the second condition is equivalent to
$\varlimsup_\alpha I(\mu_\alpha,Q_\alpha) \le I(\mu,Q)$.

\begin{theorem}
  \label{t:de} 
  The set $\ms M_0$ defined in \eqref{co} is $I$-dense  in $\ms M$. 
\end{theorem}

The proof is split in few lemmata in which we use the following
notation.  Let $A$ and $B$ be respectively a Borel subset of $E$ of
strictly positive $\lambda$ measure and a Borel subset of $E\times E$
of strictly positive $\lambda\times \lambda$ measure.  For a function
$f\in L^1(d\lambda)$, respectively a function $F\in L^1(d\lambda\times
d\lambda)$ we set
\begin{equation*}\dashint_A f := \frac 1{\lambda(A)}
\int_A\!\lambda(dx) \, f(x),\qquad
\dashiint_B F := \frac 1 {(\lambda\times\lambda)(B)}
\iint_B\!\lambda(dx) \lambda(dy)\, F(x,y).
\end{equation*}

\begin{lemma}
  \label{t:l0}
  Let
  \begin{equation}
    \label{cobis}
    \begin{split}
      \ms M_1:=\Big\{& (\mu,Q)\in\ms M \,:\: K:= \supp(\mu)\subset
      E\setminus E_0,\, \supp(Q)= K \times K,\,
      \\
      & \exists\; D_1,\ldots,D_\ell \subset K \mbox { disjoint open sets
        such that }
      \lambda\Big(K\setminus \bigcup_{i} D_i \Big) =0, 
      \\
      & d\mu
      =\sum_{i} a_i\, \id_{D_i} \, d\lambda,\;
      \; a_i>0,\:i=1,..,\ell,
      \\
      & dQ = \sum_{i,j} b_{ij}\, \id_{D_i\times D_j}
      \, d\lambda\times d\lambda,\;\; b_{ij}>0, \:i,j=1,..,\ell\Big\}.
    \end{split}
  \end{equation}
  The set $\ms M_0$ is $I$-dense in $\ms M_1$.
\end{lemma}

\begin{proof}
  Let $(\mu, Q)\in\ms M_1$ with $I(\mu, Q)<+\infty$, so that
  $Q^{(1)}=Q^{(2)}$. Denoting with $d(\cdot, \cdot)$ the distance in $E$,
 by Urysohn lemma, for each $D_i$, $i=1,..,\ell,$
  and $n\in\bb N$ there exists a continuous function $\phi^n_i\colon
  K\to [0,1]$ such that
  \begin{equation*}
    \phi^n_i(x)=\begin{cases}1& \mbox{ if }x\in \upbar D_i\\
      0 & \mbox{ if }d(x, D_i)\geq \frac 1 n,
    \end{cases}
  \end{equation*}
  where $\upbar D_i$ is the closure of $D_i$.  We define the sequence
  $(\mu_n,Q_n)$ by $d\mu_n=\varrho_n d\lambda$,
  $dQ_n=q_nd\lambda\times d\lambda$, with $\varrho_n=0$ in $E\setminus
  K$ and $q_n=0$ in $(E\times E)\setminus (K\times K)$, and
  \begin{equation*}
    \varrho_n(x) = \sum_{i}
    a_i \frac{\phi_i^n(x)}
    {\int\! d\lambda\,\phi_i^n}\,\lambda(D_i),
  \end{equation*}
  \begin{equation*}
    q_n(x,y) = \sum_{i,j} b _{ij}\,
    \frac{\phi_i^n(x)}
    {\int\! d\lambda\,\phi_i^n}\frac{\phi_j^n(y)}
    {\int\!d\lambda\,\phi_j^n}\,\lambda(D_i)\,\lambda(D_j).
  \end{equation*}
  In particular $\{(\mu_n, Q_n)\}\subset\ms M_0$, $(\mu_n, Q_n)\to(\mu, Q)$
  and, since $Q^{(1)}=Q^{(2)}$, $Q_n^{(1)}=Q_n^{(2)}$.  In view of 
  \eqref{I1}, 
  \begin{equation*}
    I(\mu_n, Q_n)=\iint_{K\times K} \lambda(dx)\,\lambda(dy)\,
    \Phi\big(q_n(x,y),\varrho_n(x)r(x)p(x,y)\big).    
  \end{equation*}
  Since $r(x)p(x,y)>0$ in $K\times K$ and $\varrho_n\geq c>0$ in $K$,
  by dominated convergence we conclude that $I(\mu_n, Q_n)\to I(\mu,
  Q)$.
\end{proof}

\begin{lemma}
  \label{t:l1}
  Let 
\begin{equation}
  \label{c1}
 \begin{split}
  \ms M_2:=\Big\{ & (\mu,Q)\in\ms M \,:\: 
  K:=\supp(\mu)\subset E\setminus E_0,\: 
  \supp(Q)= K \times K,\,
  \\
  &\mu\ll \lambda,\, Q\ll \lambda\times \lambda \Big\}.
 \end{split}
\end{equation}
The set $\ms M_1$ is $I$-dense in $\ms M_2$.
\end{lemma}

\begin{proof}
  Given an integer $n$, pick a family of disjoint open sets
  $D_1^n,..,D_n^n\subset K$ such that $\lambda\big(K\setminus \bigcup_i
  D_i^n\big)=0$ and the diameter of $D_i^n$ vanishes as $n\to\infty$,
  $i=1,\ldots,n$.
  For $(\mu,Q)\in \ms M_2$ with $Q^{(1)}=Q^{(2)}$, let $d\mu
  =\varrho\, d\lambda$ and $dQ= q\, d\lambda\times \lambda$. We define
  $d\mu_n =\varrho_n d\lambda$ and $dQ_n =q_n d\lambda\times
  d\lambda$, with $\varrho_n=0$ in $E\setminus K$, $q_n=0$ in
  $(E\times E)\setminus (K\times K)$, and
  \begin{equation*}
    \begin{array}{ll}
      \displaystyle
     \varrho_n (x) := \dashint_{D^n_i} \!\varrho &\qquad   \textrm{ if }
     x\in D^n_i,\\ 
      \displaystyle
     q_n (x,y) := \dashiint_{D^n_i\times D^n_j} \!\!q & \qquad \textrm{ if } 
       (x,y)\in D^n_i\times  D^n_j.
     \end{array}
  \end{equation*}
  In particular, since $Q^{(1)}=Q^{(2)}$, $Q_n^{(1)}=Q_n^{(2)}$.
  Moreover, $\{(\mu_n,Q_n)\}\subset \ms M_1$, and $(\mu_n,Q_n) \to (\mu,Q)$.
  In view of \eqref{I1}, 
  \begin{equation*}
    \begin{split}
      I(\mu_n,Q_n) &= \iint_{K\times K}\!\! \lambda(dx)\lambda(dy)
      \, \Phi \big(q_n(x,y),\varrho_n (x) \, r(x)p(x,y) \big)
      \\
      &= \sum_{i,j} \lambda(D^n_i) \lambda(D^n_j)\:
      \dashiint_{D^n_i\times D^n_j}\!\Big\{ 
      q_n \log\frac {q_n}{\varrho_n}+q_n\log \frac1
      {rp}-\big( q_n-\varrho_n \, {rp} \big)\Big\}.
    \end{split}
  \end{equation*}
  By convexity of the function $a\log(a/b)$, using Jensen's inequality
  we get
  \begin{equation*}
    q_n(x,y)\log\frac {q_n(x,y)}{\varrho_n(x)}
    \le  \dashiint_{D^n_i \times D^n_j}q\log\frac {q}{\varrho}, 
    \qquad (x,y)\in D^n_i\times D^n_j
 \end{equation*}
  so that
  \begin{equation*}\begin{split}
      I(\mu_n,Q_n)\leq\; & I(\mu, Q)+\sum_{i,j}\iint_{D^n_i \times
        D^n_j}\,d\lambda \times d\lambda
      \;\big(q-q_n\big)\log \frac1 {rp}\\
      & \; +\sum_{i,j}\iint_{D^n_i \times D^n_j}\,\,d\lambda \times
      d\lambda\, \big(\varrho_n-\varrho\big)\,{rp}.
    \end{split}\end{equation*}
  Since both $rp$ and $\log (1/rp)$ are continuous in $K\times
  K$, we conclude the proof taking $n\to\infty$.
\end{proof}

The next lemma is the key step and relies on the technical condition
(iv) in Assumption~\ref{t:1.1}.

\begin{lemma}
  \label{t:l2}
  Let
  \begin{equation}
    \label{c2}
    \ms M_3:=\Big\{ (\mu,Q)\in\ms M \,:\: 
    \mu\ll \lambda,\, Q\ll \lambda\times \lambda \Big\}.
  \end{equation}
  The set $\ms M_2$ is $I$-dense in $\ms M_3$.
\end{lemma}

\begin{proof}
  For $\delta>0$ let $A_\delta\subset E$ be the open set defined by
  $A_\delta:=\{x\in E: r(x)<\delta\}$.  
  Given $(\mu, Q)\in \ms M_3$, with $Q^{(1)}=Q^{(2)}$, we write
   $d\mu=\varrho\,d\lambda$ and $dQ=q\, d\lambda\times d\lambda$.  For
  $\delta>0$, we set 
  \begin{equation*}
      \varrho_\delta(x):=
      \begin{cases}
        \varrho(x) 
        & \mbox{if }\; x\in E\setminus A_{2\delta}
        \\
        \frac 1 {\lambda(A_{2\delta}\setminus A_\delta)}
        \displaystyle\int_{A_{2\delta}}\lambda(dx')\,\varrho(x')
          & \mbox{if }\; x \in A_{2\delta}\setminus A_\delta\\
          0 & \mbox{if }\; x\in A_\delta,
        \end{cases}
  \end{equation*}
  and
  \begin{equation*}
    q_\delta(x,y):=\begin{cases}
        q(x,y) & 
        \mbox{if } \; (x,y)\in \big(E\setminus A_{2\delta}\big)^2
        \\
        \frac 1 {\lambda(A_{2\delta}\setminus A_\delta)}
        \displaystyle\int_{A_{2\delta}}\lambda(dy') \,q(x,y')
        & \mbox{if }\; (x,y)\in 
        E\setminus A_{2\delta}\times A_{2\delta}\setminus A_\delta
        \\
        \frac 1 {\lambda(A_{2\delta}\setminus A_\delta)}
        \displaystyle\int_{A_{2\delta}}\lambda(dx') \,q(x',y)
        & \mbox{if }\; (x,y)\in  
        A_{2\delta}\setminus A_\delta\times E\setminus A_{2\delta}
        \\
        \frac 1 {\big( \lambda(A_{2\delta}\setminus A_\delta)\big)^2}\,
        Q\big(A_{2\delta}\times A_{2\delta}\big)
        &\mbox{if } \; (x,y)\in \big(A_{2\delta}\setminus
        A_\delta\big)^2
        \\
        0 & \mbox{if }\; x\in A_\delta\;\;
        \mbox{or }\; y\in A_\delta.
      \end{cases} 
    \end{equation*}
    By letting $d\mu_\delta:=\varrho_\delta\,d\lambda$ and
    $dQ_\delta:=q_\delta\,d\lambda\times d\lambda$, it follows that
    $(\mu_\delta,\,Q_\delta)\in \ms M_2$ and
    $(\mu_\delta,\,Q_\delta)\to (\mu,\,Q)$.  
    Moreover, since $Q^{(1)}=Q^{(2)}$,
    $Q_\delta^{(1)}=Q_\delta^{(2)}$. In view of \eqref{I1},
    \begin{equation}
      \label{Idel}
         I(\mu_\delta,\, Q_\delta)= \iint\!\! d\lambda\times
        d\lambda \, \Phi\big(q_\delta,\, \varrho_\delta\,rp\big).
     \end{equation}

     Consider first the integral over $(E\setminus A_{2\delta})^2$.
     By definition of $(\mu_\delta, Q_\delta)$,
     \begin{equation}\label{lore0}
       \lim_{\delta \downarrow 0} \iint_{(E\setminus
         A_{2\delta})^2}\!\! d\lambda\times d\lambda \:
       \Phi\big(q_\delta,\varrho_\delta\,rp\big)= I(\mu,Q). 
     \end{equation} 
     The proof of the lemma will be achieved by showing that the
     other contributions to the right hand side of \eqref{Idel} vanish
     ad $\delta\downarrow 0$.

     Consider the integral over $E\setminus A_{2\delta}\times
     A_{2\delta}\setminus A_\delta$, namely
     \begin{equation*}
       \begin{split}
         &\iint_{E\setminus A_{2\delta}\times A_{2\delta}\setminus
           A_\delta}\!\! d\lambda\times d\lambda\;
         \Phi\big(q_\delta,\,\varrho_\delta\,rp\big)
         \\
         &=\frac{\lambda(A_{2\delta})}{\lambda(A_{2\delta}\!\setminus\!
           A_\delta)}\iint_{E\setminus A_{2\delta}\times
           A_{2\delta}\setminus A_\delta}\!\!\!\! \lambda(dx) \lambda(dy)\,
         \Phi\Big(\dashint_{A_{2\delta}}
         q(x,\cdot),\,\dashint_{A_{2\delta}}\varrho(x)\,r(x)p(x,\cdot)
         \Big) +R_\delta
       \end{split}
     \end{equation*}
     where
     \begin{equation*}
       \begin{split}
         R_\delta=&\iint_{E\setminus A_{2\delta}\times
           A_{2\delta}\setminus A_\delta} \!\! \lambda(dx)\lambda(dy)\,
         \\
         &\quad \bigg\{
         \frac{\lambda(A_{2\delta})}{\lambda(A_{2\delta}\!\setminus\!{A_\delta})}
         \, \dashint_{A_{2\delta}} q(x,\cdot) 
         \log
         \frac{\lambda(A_{2\delta})}{\lambda(A_{2\delta}\!\setminus\! {A_\delta})}
         \frac{\dashint_{A_{2\delta}} r(x)p(x,\cdot)}{r(x)p(x,y)}
         \\&\qquad 
         -\Big(
         \frac{\lambda(A_{2\delta})}{\lambda(A_{2\delta}\!\setminus\!{A_\delta})}
         \, \dashint_{A_{2\delta}}\varrho(x)r(x)p(x,\cdot)
         -\varrho(x)r(x)p(x,y)\Big)\bigg\}.
       \end{split}\end{equation*}
     By choosing $\delta=\delta_n$ as in condition (iv) of
     Assumption~\ref{t:1.1} and using that $p$ is
     strictly positive, it follows that $\lim_{n}R_{\delta_n}=0$.
     Moreover, by Jensen inequality,
     \begin{equation*}
       \begin{split}
         &\frac{\lambda(A_{2\delta})}{\lambda(A_{2\delta}\!\setminus\!{A_\delta})}
         \iint_{E\setminus A_{2\delta}\times A_{2\delta}\setminus
           A_\delta}\lambda(dx)\lambda(dy)\,
         \Phi\Big(\dashint_{A_{2\delta}}
         q(x,\cdot),\,\dashint_{A_{2\delta}}\varrho(x)\,r(x)p(x,\cdot)
         \Big)\\ 
         &\qquad 
         \leq \iint_{E\setminus A_{2\delta}\times
           A_{2\delta}}\!\! d\lambda\times d\lambda\, 
         \Phi(q,\varrho\, rp).
       \end{split}
     \end{equation*}
     Hence
     \begin{equation}\label{lore1}\begin{split}
         &\varlimsup_n\iint_{E\setminus A_{2\delta_n}\times
           A_{2\delta_n}\setminus A_{\delta_n}} \!\! d\lambda\times
         d\lambda\, \Phi\big(q_{\delta_n},\varrho_{\delta_n} rp\big)
         \\
         & \leq \lim_n \iint_{E\setminus A_{2\delta_n}\times
           A_{2\delta_n}}\!\!d\lambda\times d\lambda\, \Phi(q,\varrho\,
         rp)=0
       \end{split}\end{equation}

    We next consider the integral over $A_{2\delta}\setminus
     A_\delta\times E\setminus A_{2\delta}$, namely
     \begin{equation*}
       \begin{split}
         &\iint_{A_{2\delta}\setminus A_\delta\times E \setminus
           A_{2\delta}}\!\!\!\! d\lambda\times d\lambda\:
         \Phi\big(q_\delta, \varrho_\delta\, r p\big)\\
         &=
         \frac{\lambda(A_{2\delta})}{\lambda(A_{2\delta}\!\setminus\!{A_\delta})}
         \iint_{A_{2\delta}\setminus A_\delta\times E\setminus
           A_{2\delta}}\!\! \lambda(dx)\lambda(dy)\,
         \Phi\Big(\dashint_{A_{2\delta}}
         q(\cdot,y), \dashint_{A_{2\delta}}\varrho(\cdot)r(\cdot) p(\cdot,y)
         \Big) +R_\delta,
       \end{split}
     \end{equation*}
     where
     \begin{equation*}
       \begin{split}
         R_\delta = 
         &\frac{\lambda(A_{2\delta})}{\lambda(A_{2\delta}\!\setminus\!{A_\delta})}
         \iint_{A_{2\delta}\setminus A_\delta\times E\setminus
           A_{2\delta}}\!\! \lambda(dx)\lambda(dy)\,
         \bigg\{\dashint_{A_{2\delta}}
         q(\cdot,y)\log\frac{\dashint_{A_{2\delta}}
          \varrho(\cdot)\,r(\cdot)p(\cdot,y)}
         {\dashint_{A_{2\delta}}\varrho(\cdot)\,r(x)p(x,y)}
         \\ &
         \qquad\qquad
         +
         \dashint_{A_{2\delta}}\varrho(\cdot)\,r(x)p(x,y)-
         \dashint_{A_{2\delta}}\varrho(\cdot)\,r(\cdot)p(\cdot,y)\bigg\}.
       \end{split}
     \end{equation*}
     By condition (ii) in Assumption~\ref{t:1.1}, there exists $c>0$
     such that $p(x,y)\geq c$. We thus deduce
     \begin{equation*}
       R_\delta\leq
       \frac{\lambda(A_{2\delta})}{\lambda(A_{2\delta}\!\setminus\!{A_\delta})} 
       \iint_{A_{2\delta}\setminus A_\delta\times E\setminus
         A_{2\delta}}\!\! \lambda(dx)\lambda(dy)\, 
       \dashint_{A_{2\delta}} q(\cdot,y)\log \frac {\|rp\|} { c r(x)} +
       \|rp\|\,\mu(A_{2\delta}). 
     \end{equation*} 
     By definition of the set $A_\delta$, $r(x)>\delta$ for $x\in
     A_{2\delta}\setminus{A_\delta}$. Hence
     \begin{equation*}
       \begin{split}
         R_\delta &\leq \iint_{A_{2\delta}\times E\setminus
           A_{2\delta}}\!\! \lambda(dx)\lambda(dy)\,
         q(x,y)\log\frac {\|rp\|}{c\delta} +\|rp\|\,\mu(A_{2\delta})\\
         &\leq \iint_{A_{2\delta}\times E\setminus
           A_{2\delta}}\!\! \lambda(dx)\lambda(dy)\, q(x,y)\log\frac {2\|rp\|}{c
           r(x)} +\|rp\| \,\mu(A_{2\delta}),
       \end{split}
     \end{equation*}
     where in the last inequality we used $r(x)<2\delta$ for $x\in
     A_{2\delta}$. In view of Lemma~\ref{t:1r} we conclude that
     $\varlimsup_{\delta\downarrow 0} R_\delta\leq 0$.  By using Jensen
     inequality as in the previous step we deduce
     \begin{equation}
       \label{lore2}
       \lim_{\delta\downarrow 0}\iint_{A_{2\delta}\setminus
         A_{\delta}\times E\setminus A_{2\delta}}\!\! 
       d\lambda\times d\lambda\; 
       \Phi\big(q_{\delta},\varrho_{\delta}\, rp\big) =0.
     \end{equation}

     We finally consider the integral over $(A_{2\delta}\setminus
     A_\delta)^2$, namely
     \begin{equation*}
       \begin{split}
         &\iint_{(A_{2\delta}\setminus A_\delta)^2}
         \!\! d\lambda\times d\lambda \;
         \Phi\big(q_\delta,\varrho_\delta\, rp\big)\\
         &=
         \Big(\frac{\lambda(A_{2\delta})}
           {\lambda(A_{2\delta}\!\setminus\!{A_\delta})}\Big)^2   
         \iint_{(A_{2\delta}\setminus A_\delta)^2}\!\!
         d\lambda\times d\lambda\, 
         \Phi\Big(\dashiint_{(A_{2\delta})^2}
         q,\dashiint_{(A_{2\delta})^2}
         \varrho\,rp \Big) +R_\delta,
       \end{split}
     \end{equation*}
     where
     \begin{equation*}
       \begin{split}
         R_\delta= & \iint_{(A_{2\delta}\setminus
           A_\delta)^2}\!\! \lambda(dx)\lambda(dy)
         \\
         &\quad \bigg\{
         \Big(\frac{\lambda(A_{2\delta})}
         {\lambda(A_{2\delta}\!\setminus\!{A_\delta})}\Big)^2
         \dashiint_{(A_{2\delta})^2}q
         \log \frac{\lambda(A_{2\delta})}
         {\lambda(A_{2\delta}\!\setminus\!{A_\delta})}
         \frac {\dashiint_{(A_{2\delta})^2} \varrho \,rp}
         {\dashint_{A_{2\delta}} \varrho\, r(x)p(x,y)}\\
         &\qquad +
         \frac{\lambda(A_{2\delta})}{\lambda(A_{2\delta}\!\setminus\!{A_\delta})}
         \dashint_{A_{2\delta}} \varrho\,r(x)p(x,y)
         -\Big(\frac{\lambda(A_{2\delta})}
         {\lambda(A_{2\delta}\!\setminus\!{A_\delta})}\Big)^2
         \dashiint_{(A_{2\delta})^2} \varrho\,rp
         \bigg\}.
       \end{split}
     \end{equation*}
     As in the previous step, we now use that 
     $p(x,y)\geq c $ for some $c>0$. We deduce
     \begin{equation*}
       \begin{split}
         R_\delta\leq
         &
         \Big(\frac{\lambda(A_{2\delta})}
         {\lambda(A_{2\delta}\!\setminus\!{A_\delta})}\Big)^2
         \iint_{(A_{2\delta}\setminus A_\delta)^2}
         \!\! \lambda(dx)\lambda(dy)\,
         \dashiint_{(A_{2\delta})^2} q\,
         \log \frac{\lambda(A_{2\delta})}
         {\lambda(A_{2\delta}\!\setminus\!{A_\delta})}
         \frac {\|rp\|} { c r(x)}
         \\
         &\qquad+ \|rp\|\,\mu(A_{2\delta})\,
         \lambda(A_{2\delta}\!\setminus\! A_\delta).
       \end{split}
     \end{equation*}
     Since $r(x)>\delta$ for $x\in A_{2\delta}\setminus{A_\delta}$,
     \begin{equation*}\begin{split}
         R_\delta &\leq
         \iint_{(A_{2\delta})^2}\!\!\lambda(dx)\lambda(dy)\, 
         q(x,y)\log  \frac{\lambda(A_{2\delta})}
         {\lambda(A_{2\delta}\!\setminus\!{A_\delta})}
         \frac {\|rp\|}{c\delta} + 
         \|rp\|\,\mu(A_{2\delta})\,\lambda(A_{2\delta}\!\setminus\! A_\delta)
         \\
         &\leq \iint_{(A_{2\delta})^2}\!\! \lambda(dx)\lambda(dy)\,
         q(x,y)\log\frac{\lambda(A_{2\delta})}
         {\lambda(A_{2\delta}\!\setminus \!{A_\delta})}
         \frac {2\|rp\|}{c r(x)} +
         C\,\mu(A_{2\delta})\,\lambda(A_{2\delta}\!\setminus\! A_\delta),
       \end{split}
     \end{equation*}
     where in the last inequality we used $r(x)<2\delta$ for $x\in
     A_{2\delta}$. By choosing $\delta=\delta_n$, where $\delta_n$ is
     the sequence in condition (iv) of Assumption~\ref{t:1.1} and
     using Lemma \ref{t:1r}, we deduce that $\varlimsup_{n}
     R_{\delta_n}\leq 0$.
     By using Jensen inequality as in the previous steps we conclude
     \begin{equation}
       \label{lore3}
       \lim_{n\to\infty}
       \iint_{(A_{2\delta_n}\setminus A_{\delta_n})^2}\!\! d\lambda\times d\lambda\; 
       \Phi\big(q_{\delta_n},\varrho_{\delta_n}\,rp\big)=0
     \end{equation}

     Since $\varrho_\delta$ vanishes on $A_\delta$ and $q_\delta$
     vanishes on $\big(A_\delta\times E \big)\cup \big(E\times
     A_\delta\big)$, \eqref{lore0}-\eqref{lore3} yield the statement. 
   \end{proof}

\begin{lemma} 
  \label{t:l3}
  Let 
  \begin{equation}
    \label{c3}
    \ms M_4:=\big\{ (\mu,Q)\in\ms M \,:\: 
    \mu
\perp\lambda,
\, Q=0 \big\},
  \end{equation}
  The set $\ms M_3$ is $I$-dense in $\ms M_4$.
\end{lemma}

\begin{proof}
  Consider a Borel partition $E=\bigcup_{i}D_i^{n}$ such that
  $\lambda(D_i^{n})>0$ and the diameter of $D_i^{n}$ vanishes as
  $n\to\infty$. 
Given $\mu\perp\lambda$,
we set $d\mu_n=\varrho_n d\lambda$, where
  \begin{equation*}
    \varrho_n(x)=\frac{\mu\big(D_i^{n}\big)}{\lambda\big(D_i^{n}\big)},
    \qquad x\in D_i^{n} 
  \end{equation*}
  and $Q_n=0$. In particular $(\mu_n, Q_n)\in\ms M_3$ and 
  $(\mu_n, Q_n)\to (\mu, 0)$. Moreover
  \begin{equation*}
    \lim_{n\to\infty} I(\mu_n,Q_n)=\lim_{n\to\infty}\iint
    \lambda(dx)\,\lambda(dy) \varrho_n(x)r(x)p(x,y)=\mu(r)= 
    I(\mu,0)
  \end{equation*}
  where we used the representation \eqref{I1} for $I(\mu,Q)$.
\end{proof}

\begin{proof}[Proof of Theorem~\ref{t:de}]
  Let $(\mu,Q)\in \ms M$ be such that $I(\mu,Q)<+\infty$. We decompose
$\mu$ into the absolutely continuous and singular parts with respect to $\lambda$, i.e.
  $\mu=\mu_\mathrm{ac} +\mu_\mathrm{s}$ and we recall that by
  Lemma~\ref{t:rl} $Q\ll Q^{\mu_\mathrm{ac}}$.  In particular, letting
  $\alpha = \mu_\mathrm{ac}(E)$,
  \begin{equation*}
    (\mu,Q)= \alpha \, \Big( \frac 1\alpha \mu_\mathrm{ac}, \frac
    1\alpha Q\Big) 
    + (1-\alpha) \Big( \frac 1{1-\alpha} \mu_\mathrm{s}, 0\Big).
  \end{equation*}
  In view of Lemmata \ref{t:l0}--\ref{t:l2} there
  exists a sequence $\{(\mu_{1,n},Q_{1,n})\} \subset \ms M_0$ such
  that $(\mu_{1,n},Q_{1,n}) \to ( \alpha^{-1} \mu_\mathrm{ac},
  \alpha^{-1} Q\big)$ and $I(\mu_{1,n},Q_{1,n}) \to I\big( \alpha^{-1}
  \mu_\mathrm{ac}, \alpha^{-1} Q\big)$.  Moreover, by Lemmata
  \ref{t:l0}--\ref{t:l3}, there exists a sequence
  $\{(\mu_{2,n},Q_{2,n})\} \subset \ms M_0$ such that
  $(\mu_{2,n},Q_{2,n}) \to ( (1-\alpha)^{-1} \mu_\mathrm{s}, 0\big)$
  and $I(\mu_{2,n},Q_{2,n}) \to I\big( (1-\alpha)^{-1}
  \mu_\mathrm{s},0\big)$.

  The sequence $\{\alpha (\mu_{1,n},Q_{1,n}) +
  (1-\alpha)(\mu_{2,n},Q_{2,n})\}$ is in $\ms M_0$ and converges to
  $(\mu,Q)$. By the convexity of $I$,
  \begin{equation*}
    I\big( \alpha (\mu_{1,n},Q_{1,n}) + (1-\alpha)(\mu_{2,n},Q_{2,n})
    \big)
    \le \alpha I (\mu_{1,n},Q_{1,n}) + (1-\alpha) I (\mu_{2,n},Q_{2,n})
  \end{equation*}
  so that  
  \begin{equation*}
    \begin{split}
    &\varlimsup_n I\big( \alpha (\mu_{1,n},Q_{1,n}) +
    (1-\alpha)(\mu_{2,n},Q_{2,n})\big) 
    \\
    &\qquad \le 
    \alpha I ( \alpha^{-1} \mu_\mathrm{ac}, \alpha^{-1} Q\big)
    + (1-\alpha) I ( (1-\alpha)^{-1} \mu_\mathrm{s},0\big)
    = I(\mu,Q)
    \end{split}
  \end{equation*}
  where we used the representation \eqref{I1} in the last equality.
\end{proof}

\begin{proof}[Proof of Theorem~\ref{t:ldp} (conclusion).]
  The upper bound follows from Proposition~\ref{t:ub} and
  Lemma~\ref{t:rl}, which also yields the convexity and lower
  semi-continuity of $I$. 
  Recalling \eqref{Jrf}, Lemma~\ref{t:rlb} and Proposition~\ref{t:lbns}
  imply the uniform lower bound with rate function $\sce J$. In view
  of the lower semi-continuity of $I$ and Theorem~\ref{t:de} we
  conclude $ \sce J =I$.  Finally, the goodness of the rate function
  $I$ follows from the exponential tightness proven in
  Lemma~\ref{t:et} and \cite[Lemma~1.2.18]{DZ}.
\end{proof}

\section{Projections}
\label{sec:cor}

\subsection*{Large deviations of the empirical measure}
In the context of irreducible finite state Markov chain, the
representation of the Donsker-Varadhan functional in terms of $I$ has
been obtained in \cite{BP,KL}. This result has been proven for
countable state space in \cite{BFG2}. The proof presented below relies
on the variational representation of Lemma~\ref{t:rl} and on the
Sion's minimax theorem. It takes advantage of the compactness of $E$.

\begin{proof}[Proof of Corollary~\ref{t:ldem}]
  Let $I_1\colon \mc M_1(E)\to [0,+\infty]$ be the functional
  \begin{equation}
    \label{pI}
    I_1(\mu) := \inf_{Q\in \mc M_+(E\times E)} I(\mu,Q). 
  \end{equation}
  By contraction principle and Theorem~\ref{t:ldp}, as $T\to +\infty$
  the family $\big\{\bb P_x\circ \mu_T^{-1}\big\}_{T>0}$ satisfy a
  large deviation principle with rate function $I_1$. To complete the
  proof it is therefore enough to show $\hat I =I_1$.

  We first prove the inequality $I_1\ge \hat I$.  In view of \eqref{I}
  we can restrict the infimum on the right hand side of \eqref{pI} to
  elements $Q\in \mc M_+(E\times E)$ satisfying $Q^{(1)}=Q^{(2)}$.
  For such elements, by the variational characterization of $I$ proven in
  Lemma~\ref{t:rl},
  \begin{equation*}
    I(\mu,Q) = \sup_{F\in C(E\times E)}  
    \big\{ Q(F) - \mu\big(r^F -r\big) \big\}.
  \end{equation*}
  Fix $f\in C(E)$ and choose $F(x,y)=f(y)-f(x)$, $(x,y)\in E\times E$.
  Since $Q^{(1)}=Q^{(2)}$,
  \begin{equation*}
    I(\mu,Q) \ge - \iint \!\! \mu(dx) c(x,dy)\, 
    \big[ e^{f(y)-f(x)} -1\big].
  \end{equation*}
  As the right hand side does not depend on $Q$ we deduce
  \begin{equation*}
    I_1(\mu) \ge - \iint \!\! \mu(dx) c(x,dy)\, 
    \big[ e^{f(y)-f(x)} -1\big]
  \end{equation*}
  and the result follows by optimizing on $f$.

  We next prove the inequality $I_1\le \hat I$.  Fix $\mu\in \mc
  M_1(E)$ and observe that $I_1(\mu) \le I(\mu,0)=\mu(r)<+\infty$.
  By Lemma~\ref{t:rl},
  \begin{equation*}
    I_1(\mu) =\inf_{Q}\:  \sup_{\phi,F } \: \Gamma_\mu (Q,\phi,F)
  \end{equation*}
  where the infimum is carried out over all $Q\in  \mc M_+(E\times
  E)$, the supremum over all $(\phi,F)\in C(E)\times C(E\times E)$, and 
  $\Gamma_\mu\colon \mc M_+(E\times E)\times C(E)\times C(E\times E)\to
  \bb R$ is the continuous functional defined by
  \begin{equation*}
    \Gamma_\mu (Q,\phi,F) 
    = Q^{(1)} (\phi) - Q^{(2)}(\phi) + Q(F) - \mu (r^F -r).
  \end{equation*}
  As follows from a direct application of H\"older inequality, the map
  $F\mapsto \mu (r^F)$ is convex. Hence, for each $Q$ the map
  $(\phi,F)\mapsto\Gamma_\mu (Q,\phi,F)$ is concave.  Since for each
  $(\phi,F)$ the map $Q\mapsto \Gamma_\mu (Q,\phi,F)$ is affine, we
  would like to apply the Sion's minimax theorem to get
  \begin{equation}
    \label{minmax}
    I_1(\mu) = \sup_{\phi,F } \: \inf_{Q}\: 
    \Gamma_\mu (Q,\phi,F).
  \end{equation}
  Since neither $\mc M_+(E\times E)$ nor $C(E)\times C(E\times E)$ is 
  compact, \eqref{minmax} needs however to be justified. Postponing this
  step, we first conclude the argument. By choosing on the right hand
  side of \eqref{minmax} $Q(dx,dy) = \mu(dx) c(x,dy) e^{F(x,y)}$ we
  get 
  \begin{equation*}
    \begin{split}
    I_1(\mu) &\le \sup_{\phi,F} \iint\!\! \mu(dx) c(x,dy) \big[
    e^{F(x,y)} \big(F(x,y) + \phi(x) -\phi(y) \big) 
    -\big( e^{F(x,y)} -1\big) \big] \\
    &\le  \sup_{\phi} \iint\!\! \mu(dx) c(x,dy) 
    \sup_{\lambda \in \bb R} 
    \big[e^\lambda \big(\lambda + \phi(x) -\phi(y)\big)  
    - \big(e^{\lambda} -1\big) \big]\\
    &=  \sup_{\phi} \Big\{ -\iint\!\! \mu(dx) c(x,dy) 
    \, \big[ e^{\phi(y) -\phi(x)} -1\big]  \Big\}= \hat I(\mu).
    \end{split}
  \end{equation*}

  We are left with the proof of \eqref{minmax}. To this end we apply
  the generalization of the Sion minimax theorem proven in \cite{H}
  that states the following. Under the continuity and
  convexity/concavity assumptions discussed before, a sufficient
  condition for the minimax identity
  \begin{equation*}
    \inf_{Q}\:  \sup_{\phi,F } \: \Gamma_\mu (Q,\phi,F) 
    = \sup_{\phi,F } \: \inf_{Q}\:  \Gamma_\mu (Q,\phi,F)
  \end{equation*}
  is that there exist a nonempty convex compact 
  $K\subset C(E) \times C(E\times E)$ and a compact 
  $\mc H \subset \mc M_+(E\times E)$ such that 
  \begin{equation}
    \label{scmm}
     \inf_{Q}\:  \sup_{\phi,F } \: \Gamma_\mu (Q,\phi,F) 
     \le  \inf_{Q\not\in \mc H } \: 
     \sup_{(\phi,F) \in K } \: \Gamma_\mu(Q,\phi,F).
  \end{equation}
  We choose $\mc H =\big\{ Q\in \mc M_+(E\times E):\, \|Q\|_{\mathrm{TV}}\le h
  \big\}$ (here $\|Q\|_{\mathrm{TV}}$ is the total mass of $Q$) for some $h>0$ to be
  fixed later and let $K$ be the singleton $K=\{ (0,1)\}$. If
  $Q\not\in \mc H$ then 
  \begin{equation*}
    \Gamma_\mu(Q,0,1) = Q(1) - (e-1) \mu(r) \ge h - (e-1) \mu(r).
  \end{equation*}
  Since, as already observed, 
  \begin{equation*}
    \inf_{Q}\sup_{\phi,F }\Gamma_\mu (Q,\phi,F)  \le I(\mu,0) = \mu(r), 
  \end{equation*}
  by choosing $h\ge e \mu(r)$ the condition \eqref{scmm} holds and we
  have concluded the proof of \eqref{minmax}.
\end{proof}

\subsection*{Large deviations of the empirical flow}

\begin{proof}[Proof of Corollary~\ref{t:ldem2}]
  Let $I_2\colon \mc M_+(E\times E)\to [0,+\infty]$ be the functional
  \begin{equation}
    \label{pI2}
    I_2(Q) := \inf_{\mu\in \mc M_1(E)} I(\mu,Q). 
  \end{equation}
  By contraction principle and Theorem~\ref{t:ldp}, as $T\to +\infty$
  the family $\big\{\bb P_x\circ Q_T^{-1}\big\}_{T>0}$ satisfy a
  large deviation principle with rate function $I_2$. To complete the
  proof it is therefore enough to show $\tilde I =I_2$.

  We first prove the inequality $I_2\ge \tilde I$. We use the variational characterization \eqref{I=I}
restricting to $Q$ with equal marginals. Given $\alpha\in (-r_{\mathrm m}, +\infty)$, we chose 
$$
F(x,y)=\log\Big[\frac{Q(dx, dy)}{Q(dx, E)c(x, dy)}(r(x)+\alpha)   \Big].
$$
By direct computations $Q^\mu(e^F-1)=\alpha$, so that
\begin{equation*}\begin{split}
 I(\mu,Q)&\geq Q(F)-Q^{\mu}\big( e^F-1 \big)\\
&=\iint Q(dx,dy)\log\Big[\frac{Q(dx, dy)}{Q(dx, E)c(x, dy)}(r(x)+\alpha) 
  \Big]\;-\alpha.
\end{split}\end{equation*}
The result follows by optimizing over $\alpha$.
We observe the  choice of $F$ is not really legal since it could be not continuous, however a truncation procedure similar to the one in Lemma
\ref{t:rl} leads to the same conclusion.

We next prove $I_2\leq \tilde I$. By definition of $\tilde I$ we can assume that   $Q$ has equal marginals.
Given $Q$, if there exists $\alpha>-r_{\mathrm m}$
such that $\iint Q(dx, E)/(r(x)+\alpha)=1$ we chose
$$\mu^Q(dx)=\frac {Q(dx, E)}{r(x)+\alpha},$$
then $I_2(Q)\leq I(\mu^Q, Q)$. By a direct computation 
$$I(\mu^Q, Q)=\iint Q(dx, dy)\log \Big[\frac {Q(dx, dy)}{Q(dx, E)c(x, dy)}(r(x)+\alpha)   \Big]-\alpha\leq \tilde I (Q).$$
If such $\alpha$ does not exists, by monotone convergence
$$
\int \frac {Q(dx, E)}{r(x)-r_{\mathrm m}}\leq 1
$$
and we can chose
$$\mu^Q(dx)=\frac {Q(dx, E)}{r(x)-r_{\mathrm m}} + \Big(1-\iint \frac {Q(dz, E)}{r(z)-r_{\mathrm m}} \Big)\delta_{x_0}(dx),$$
with $x_0$  such that $r(x_0)=r_{\mathrm m}$. From $\eqref{I1}$ by direct computation we obtain
$$
I(\mu^Q, Q)=\iint Q(dx, dy)\log \Big[\frac {Q(dx, dy)}{Q(dx, E)c(x, dy)}(r(x)-r_\mathrm m)   \Big]+r_{\mathrm m} 
\leq \tilde I (Q),$$
where the last inequality follows  by monotone convergence.
\end{proof}

\subsection*{Acknowledgements} 
We thank the referee for suggesting us to formulate result for a general (closed) set $E_0$.

\end{document}